\documentclass[11pt, reqno]{amsart}

\usepackage{tikz}
\usepackage{graphicx}

\usepackage{indentfirst}
\usepackage{amsmath}
\usepackage{amsthm}
\usepackage{amssymb}
\usepackage{amsfonts}

\usepackage[colorlinks=true, allcolors=blue]{hyperref}
\usepackage[top=1.5in, bottom=1.5in, left=1.5in, right=1.5in]{geometry}

\theoremstyle{plain}
\newtheorem{lemma}{Lemma}[section]
\newtheorem{corollary}[lemma]{Corollary}

\newtheorem{theorem}[lemma]{Theorem}

\theoremstyle{definition}

\theoremstyle{remark}
\newtheorem{remark}[lemma]{Remark}

\numberwithin{equation}{section}

\usepackage{url}
\newcommand\numberthis{\addtocounter{equation}{1}\tag{\theequation}}

\newcommand{\bvec}[1]{\mathbf{#1}}
\newcommand{\RR}[0]{\mathbb{R}}
\newcommand{\ZZ}[0]{\mathbb{Z}}
\newcommand{\CC}[0]{\mathbb{C}}

\newcommand{\kronsym}[2]{\left( \dfrac{{#1}}{{#2}}\right)}
\newcommand{\e}[1]{\mathrm{e}\!\left({#1}\right)}
\newcommand{\gausssum}[2]{g \! \left({#1};{#2}\right)}

\newcommand{\Vol}[0]{\operatorname{Vol}}

\renewcommand{\pmod}[1]{\:\left(\operatorname{mod} {#1}\right)}

\begin{document}
\title[Local Densities of Diagonal Ternary Quadratic Forms]{Local Densities of Diagonal Integral Ternary Quadratic Forms at Odd Primes}
\author{Edna Jones}
\address{Mathematics Department, Rutgers, The State University of New Jersey, 110 Frelinghuysen Road\\
Piscataway, NJ 08854-8019, USA\\
\url{elj44@math.rutgers.edu}}
\dedicatory{Dedicated to Bruce Berndt on the occasion of his 80th birthday}
\maketitle

\begin{abstract}
We give formulas for local densities of diagonal integral ternary quadratic forms at odd primes.  Exponential sums and quadratic Gauss sums are used to obtain these formulas. These formulas (along with 2-adic densities and Siegel's mass formula) can be used to compute the representation numbers of certain ternary quadratic forms.
\end{abstract}

%

\section{Introduction}

Mathematicians have been interested in sums of squares and quadratic forms for several centuries.
For example, Hilbert stated in his 11th problem, ``[S]olve a given quadratic equation with algebraic numerical coefficients in any number of variables by integral or fractional numbers belonging to the algebraic realm of rationality determined by the coefficients'' \cite[p.~458]{HilbertProbs}.

In this paper, we add to our understanding of quadratic forms by giving formulas for local densities of diagonal integral ternary quadratic forms at odd primes. By a diagonal integral ternary quadratic form, we mean a quadratic form of the form $a x^2 + b y^2 + c z^2$, where $a, b, c$ are integers. Let $Q$ be the diagonal integral ternary quadratic form $Q(\bvec{v}) = a x^2 + b y^2 + c z^2$, where $a$, $b$, and $c$ are integers and $\bvec{v} = (x,y,z)^\top$.

Let $m$ be an integer. A necessary condition for $Q(\bvec{v}) = m$ to have a solution in the integers is that the congruence 
\begin{align} \label{cong:Qv=m}
Q(\bvec{v}) = a x^2 + b y^2 + c z^2 \equiv m \pmod{n}
\end{align}
has a solution for every positive integer $n$. This leads us to count the number of solutions to \eqref{cong:Qv=m}. For a positive integer $n$, we define the \emph{local representation number $r_{n}(m, Q)$} as 
\begin{align*}
r_{n}(m, Q) = \# \left\{ \bvec{v} \in ( \ZZ / n \ZZ )^3 : Q(\bvec{v}) \equiv m \pmod{n} \right\}.
\end{align*}

The Chinese Remainder Theorem tells us that $r_{n_1 n_2}(m,Q) = r_{n_1}(m,Q) r_{n_2}(m,Q)$ whenever $n_1$ and $n_2$ are relatively prime, so it suffices to study the local representation numbers in which $n$ is a prime power.
For a prime $p$, we use a $p$-adic density to encode information about $r_{p^k}(m, Q)$ for all positive integers $k$. Given a prime $p$, let $\ZZ_p$ denote the set of $p$-adic integers with the usual Haar measure, and define the \emph{local (representation) density} $\alpha_p (m, Q)$ at the prime $p$ by 
\begin{align*}
\alpha_p (m, Q) = \lim_{U \to \{m\}} \frac{\Vol_{\ZZ_p^3} (Q^{-1}(U))}{\Vol_{\ZZ_p}(U)},
\end{align*}
where $U$ runs over a sequence of open sets in $\ZZ_p$ containing $m$, $\Vol_{\ZZ_p^3} (Q^{-1}(U))$ is the volume of $Q^{-1}(U)$ in $\ZZ_p^3$, and $\Vol_{\ZZ_p}(U)$ is the volume of $U$ in $\ZZ_p$. The local density $\alpha_p (m, Q)$ is also called a \emph{$p$-adic density}.
It can be shown that
\begin{align} \label{eq:repntodensity}
\alpha_p (m, Q) = \lim_{k \to \infty} \frac{r_{p^k}(m , Q)}{p^{2k}}.
\end{align}
(See \cite[p.~368]{HankeLocalDensities} and \cite[p.~37, Lemma~2.7.2]{HankeQuadForms}.)

As can be seen by \eqref{eq:repntodensity}, a $p$-adic density encodes some information about $r_{p^k}(m , Q)$ for all positive integers $k$. In fact, in some cases, local densities are almost enough to compute the number of integral solutions to $Q(\bvec{v}) = m$. Define the \emph{representation number $r(m, Q)$} by
\begin{align*}
r(m, Q) = \# \left\{ \bvec{v} \in \ZZ^3 : Q(\bvec{v}) = m \right\}.
\end{align*}
In 1935, Siegel~\cite{SiegelQuadForm} proved a formula (now often called Siegel's mass formula) that can be used to compute the weighted average of the representation numbers of certain positive definite quadratic forms. We state a specialized version of Siegel's mass formula here.
\begin{theorem}[Specialized Version of Siegel's Mass Formula] \label{thm:SiegelMassFormula}
Let $m$ be an integer and $Q$ be a positive definite quadratic form of rank $3$. Let $\{ Q_j \}$ be a complete set representatives for classes in the same genus as $Q$. Then
\begin{align}
\dfrac{\displaystyle{\sum_j} \dfrac{r(m, Q_j)}{\# O(Q_j)}}{\displaystyle{\sum_j} \dfrac{1}{\# O(Q_j)}} = \alpha_\RR(m, Q) \prod_{p \text{ prime}} \alpha_p (m, Q) , \label{eq:SiegelMassFormula}
\end{align}
where $O(Q_j)$ is the orthogonal group of $Q_j$ over $\ZZ$ and $\alpha_\RR(m, Q)$ is the local density at the archimedean place. More explicitly, $\displaystyle{\alpha_\RR(m, Q) = \lim_{U \to \{m\}} \frac{\Vol_{\RR^3} (Q^{-1}(U))}{\Vol_{\RR}(U)}}$, where $U$ runs over a sequence of open sets in $\RR$ containing $m$, $\Vol_{\RR^3} (Q^{-1}(U))$ is the volume of $Q^{-1}(U)$ in $\RR^3$, and $\Vol_{\RR}(U)$ is the volume of $U$ in $\RR$.
\end{theorem}
For more information about Siegel's mass formula in English, see \cite{SiegelLectures}.

For an appropriate quadratic form $Q$, Theorem~\ref{thm:SiegelMassFormula} implies that the representation number of $r(m, Q)$ is zero (that is, there is no solution to $Q(\bvec{v}) = m$) if there is a local density that is zero. Therefore, showing a local density is zero is a method of proving that $Q(\bvec{v}) = m$ does not have a solution.

A number of formulas for local densities have been stated over the years. In 1998, Yang~\cite{YangLocalDensities} computed some local densities, but some work is needed to determine whether one of his formulas is equal to zero. In 2004, Hanke~\cite{HankeLocalDensities} did some computations of local representation numbers, but it also can be difficult to tell when some of his formulas are equal to zero due to their recursive nature. The formulas for the local densities that we will obtain can be easily computed with only knowing the prime factorizations of $a$, $b$, $c$, and $m$ and the values of certain Legendre symbols. Therefore, it is relatively easy for one to state when one of these formulas is zero.

Some other formulas for local densities are explicit but are not very general when it comes to diagonal ternary quadratic forms. In 1935, Siegel~\cite[Section 3]{SiegelQuadForm} computed some local densities. Two of Siegel's lemmas~\cite[Hilfssatz~12 and Hilfssatz~13, pp.~539--542]{SiegelQuadForm} imply that, if $p$ is a prime that does not divide $2abcm$, then $\alpha_p (m, Q) = 1 + \dfrac{1}{p} \kronsym{-abcm}{p}$, where $\kronsym{\cdot}{p}$ is the Legendre symbol. We generalize this result so that the odd prime $p$ can divide $2abcm$.
Our results also generalize to any diagonal integral ternary quadratic form the following theorem proved by Berkovich and Jagy. 
\begin{theorem}[Theorem 3.1, p.~262, \cite{BerkovichJagySum3Squares}] \label{thm:BJthm}
Let $p$ be an odd prime and $u$ be any integer with $\kronsym{-u}{p} = -1$. Let $Q(\bvec{v}) = u x^2 + p y^2 + u p z^2$. Suppose $m$ is a nonzero integer and $m = m_0 p^{m_1}$, where $\gcd(m_0, p) = 1$.  Then 
\begin{align*}
\alpha_p (m, Q) &= 
\begin{cases}
p^{-m_1/2} \left( 1 - \kronsym{- m_0}{p} \right) , &\text{if $m_1$ is even,} \\
p^{(-m_1 + 1)/2} \left( 1 + \dfrac{1}{p} \right) , &\text{if $m_1$ is odd.}
\end{cases}
\end{align*}
\end{theorem}

In this paper, only local densities at odd primes are computed since computing $2$-adic densities of quadratic forms tends to be more complicated. Our main result is the following theorem.
\begin{theorem} \label{thm:densities}
Let $Q$ be the integral quadratic form $a x^2 + b y^2 + c z^2$, where $a$, $b$, and $c$ are integers.
Let $p$ be an odd prime. Suppose $p \nmid a$, $b = b_0 p^{b_1}$, and $c = c_0 p^{c_1}$, where $b_1 \le c_1$, $\gcd(b_0, p) = 1$, and $\gcd(c_0, p) = 1$. 

Suppose $m$ is a nonzero integer and $m = m_0 p^{m_1}$, where $\gcd(m_0, p) = 1$. 

If $m_1 < b_1$, then
\begin{align} \label{eq:density m1<b1}
\alpha_p (m, Q) = 
\begin{cases}
p^{m_1/2} \left( 1 + \kronsym{a m_0}{p} \right) , &\text{if $m_1$ is even,} \\
0, &\text{if $m_1$ is odd.}
\end{cases}
\end{align}

If $b_1 \le m_1 < c_1$, then
\begin{align} \label{eq:density b1<=m1<c1}
\alpha_p (m, Q) = 
\begin{cases}
p^{b_1 / 2} \left(1 - \dfrac{1}{p} \kronsym{-a b_0}{p}^{m_1 + 1} + \left( 1 - \dfrac{1}{p} \right) \left( \dfrac{m_1 - b_1}{2} + \dfrac{(-1)^{m_1} -1}{4} \right.\right. \\
	\qquad\quad\left.\left. + \kronsym{- a b_0}{p} \left( \dfrac{m_1 - b_1}{2} + \dfrac{1 - (-1)^{m_1}}{4} \right) \right) \right), \qquad\text{if $b_1$ is even,} \\
p^{(b_1 - 1) / 2} \left(1 + \kronsym{a}{p}^{m_1 + 1} \kronsym{b_0}{p}^{m_1} \kronsym{m_0}{p} \right), \qquad\qquad\quad\text{if $b_1$ is odd.}
\end{cases}
\end{align}

If $m_1 \ge c_1$ and $b_1$ is even, then 
\begin{align} \label{eq:density m1>=c1, b1 even}
\alpha_p (m, Q) &= 
\begin{cases}
p^{b_1 / 2} \left( 1 + \dfrac{1}{p} + p^{- m_1/2 + c_1/2 - 1} \left( \kronsym{- a b_0 c_0 m_0}{p} - 1 \right) \right. \smallskip\\
	\qquad\qquad\qquad\left.+ \left( 1 - \dfrac{1}{p} \right) \left( \dfrac{c_1 - b_1}{2} + \kronsym{- a b_0}{p} \dfrac{c_1 - b_1}{2} \right) \right) , \\
	\qquad\qquad\qquad\qquad\qquad\qquad\qquad\qquad\text{if $c_1$ and $m_1$ are even,} \\
p^{b_1 / 2} \left( \left( 1 + \dfrac{1}{p} \right) \left( 1 - p^{- (m_1 + 1)/2 + c_1/2}\right) \right.\\
	\qquad\qquad\qquad\left.+ \left( 1 - \dfrac{1}{p} \right) \left( \dfrac{c_1 - b_1}{2} + \kronsym{- a b_0}{p} \dfrac{c_1 - b_1}{2} \right) \right), \\
	\qquad\qquad\qquad\qquad\qquad\qquad\qquad\qquad\text{if $c_1$ is even and $m_1$ is odd,} \\
p^{b_1 / 2} \left( 1 - p^{- m_1/2 + (c_1 - 1)/2} \kronsym{- a b_0}{p} \left( 1 + \dfrac{1}{p} \right) + \dfrac{1}{p} \kronsym{-a b_0}{p} \right. \smallskip\\
	\qquad\qquad\qquad\left.+ \left( 1 - \dfrac{1}{p} \right) \left( \dfrac{c_1 - b_1 - 1}{2} + \kronsym{- a b_0}{p} \dfrac{c_1 - b_1 + 1}{2} \right) \right), \\
	\qquad\qquad\qquad\qquad\qquad\qquad\qquad\qquad\text{if $c_1$ is odd and $m_1$ is even,} \\
p^{b_1 / 2} \left( 1 + p^{- (m_1 + 1)/2 + (c_1 - 1)/2} \left( \kronsym{c_0 m_0}{p} - \kronsym{-a b_0}{p} \right) + \dfrac{1}{p} \kronsym{-a b_0}{p}  \right.\smallskip\\
	\qquad\left.+ \left( 1 - \dfrac{1}{p} \right) \left( \dfrac{c_1 - b_1 - 1}{2} + \kronsym{- a b_0}{p} \dfrac{c_1 - b_1 + 1}{2} \right) \right), \\
	\qquad\qquad\qquad\qquad\qquad\qquad\qquad\qquad\text{if $c_1$ and $m_1$ are odd.}
\end{cases}
\end{align}

If $m_1 \ge c_1$ and $b_1$ is odd, then 
\begin{align} \label{eq:density m1>=c1, b1 odd}
\alpha_p (m, Q) &= 
\begin{cases}
p^{(b_1 - 1)/ 2} \left( 1 + \kronsym{-a c_0}{p} - p^{- m_1/2 + c_1/2} \left( 1 + \dfrac{1}{p} \right) \kronsym{- a c_0}{p} \right) , \\
	\qquad\qquad\qquad\qquad\qquad\qquad\qquad\qquad\text{if $c_1$ and $m_1$ are even,} \\
p^{(b_1 - 1)/ 2} \left( 1 + \kronsym{-a c_0}{p}  +  p^{- (m_1 + 1)/2 + c_1/2} \left( \kronsym{b_0 m_0}{p} - \kronsym{-a c_0}{p} \right) \right), \\
	\qquad\qquad\qquad\qquad\qquad\qquad\qquad\qquad\text{if $c_1$ is even and $m_1$ is odd,} \\
p^{(b_1 - 1)/ 2} \left( 1 + \kronsym{- b_0 c_0}{p}  + p^{- m_1/2 + (c_1 - 1)/2} \left( \kronsym{a m_0}{p} - \kronsym{- b_0 c_0}{p} \right)  \right) , \\
	\qquad\qquad\qquad\qquad\qquad\qquad\qquad\qquad\text{if $c_1$ is odd and $m_1$ is even,} \\
p^{(b_1 - 1)/ 2} \left( 1 + \kronsym{- b_0 c_0}{p} - p^{(-m_1 +  c_1)/2} \left( 1 + \dfrac{1}{p} \right) \kronsym{- b_0 c_0}{p} \right) , \\
	\qquad\qquad\qquad\qquad\qquad\qquad\qquad\qquad\text{if $c_1$ and $m_1$ are odd.}
\end{cases}
\end{align}

Furthermore, 
\begin{align} \label{eq:density0}
\alpha_p (0, Q) &= 
\begin{cases}
p^{b_1/2} \left( 1 + \dfrac{1}{p} + \left( 1 - \dfrac{1}{p} \right) \left( \dfrac{c_1 - b_1}{2} + \kronsym{- a b_0}{p} \dfrac{c_1 - b_1}{2} \right) \right), \\
	\qquad\qquad\qquad\qquad\qquad\qquad\qquad\qquad\text{if $b_1$ and $c_1$ are even,} \\
p^{b_1/2} \left( 1 + \dfrac{1}{p} \kronsym{-a b_0}{p} \right. \\
	\qquad\qquad\left.+ \left( 1 - \dfrac{1}{p} \right) \left( \dfrac{c_1 - b_1 - 1}{2} + \kronsym{- a b_0}{p} \dfrac{c_1 - b_1 + 1}{2} \right) \right), \\
	\qquad\qquad\qquad\qquad\qquad\qquad\qquad\qquad\text{if $b_1$ is even and $c_1$ is odd,} \\
p^{(b_1 - 1)/2} \left(1 + \kronsym{-a c_0}{p} \right) , \qquad\qquad\quad\text{if $b_1$ is odd and $c_1$ is even,} \smallskip\\
p^{(b_1 - 1)/2} \left( 1 + \kronsym{- b_0 c_0}{p} \right) , \qquad\qquad\quad\text{if $b_1$ and $c_1$ are odd.} 
\end{cases}
\end{align}
\end{theorem}

\begin{remark}
Theorem~\ref{thm:densities}, namely \eqref{eq:density m1<b1} and \eqref{eq:density b1<=m1<c1}, can be used to compute $p$-adic densities ($p$ odd) of binary quadratic forms of the form $a x^2 + b y^2$ by taking $c_1$ to be infinity for all odd primes $p$.
\end{remark}

If  $Q$ is positive definite and in a genus containing exactly one class, then the left-hand side of Siegel's mass formula \eqref{eq:SiegelMassFormula} simplifies to the representation number $r(m, Q)$.
In 1939, Jones and Pall~\cite[p.~167]{JonesPallRegPosTernQuadForms} proved that there are 82 primitive quadratic forms of the form $a x^2 + b y^2 + c z^2$ with $0 < a \le b \le c$ such that each is in a genus containing exactly one class. In 1971, Lomadze~\cite{LomadzeRep82TernaryQuadForms} computed the representation numbers for these 82 quadratic forms using the singular series.
In the future, the author would like to compute 2-adic densities and use these in conjunction with Theorem~\ref{thm:densities} to evaluate the right-hand side of Siegel's mass formula \eqref{eq:SiegelMassFormula} explicitly. In addition to giving another proof of Lomadze's evaluation of the representation numbers of the 82 primitive quadratic forms mentioned by Jones and Pall, this would give an explicit formula for the weighted average appearing in Theorem~\ref{thm:SiegelMassFormula}.

Traditionally, Hensel's lemma has been used to compute local densities. (For example, see \cite{HankeLocalDensities}.) In fact, Hensel's lemma can be used to show that the sequence 
\begin{align}
\left\{ \frac{r_{p^k}(m , Q)}{p^{2k}} \right\}_{k \ge 1} \label{localdensityseq}
\end{align}
is eventually constant if $m \ne 0$; that is, if $m \ne 0$, then there exists a $K$ such that if $k \ge K$, then 
\begin{align*}
\frac{r_{p^k}(m , Q)}{p^{2k}} = \frac{r_{p^{K}}(m , Q)}{p^{2 K}}.
\end{align*}
Without using Hensel's lemma, we show in this paper that the sequence in \eqref{localdensityseq} is eventually constant if $p$ is odd and $m \ne 0$. If $m \ne 0$ and $p$ is odd, our proofs show that the subsequence 
\begin{align*}
\left\{ \frac{r_{p^k}(m , Q)}{p^{2k}} \right\}_{k \ge m_1 + 1}
\end{align*}
is a constant sequence, where $m_1$ is the same as in Theorem~\ref{thm:densities}.

The remainder of this paper sets up and gives a proof of Theorem~\ref{thm:densities}. Section~\ref{section:LocalRepnNumsExpSumsGaussSums} relates local representation numbers to exponential sums and quadratic Gauss sums. Section~\ref{section:SuppLems} states some lemmas that are useful in the proof of Theorem~\ref{thm:densities}. A proof of Theorem~\ref{thm:densities} is given in Section~\ref{section:Proof}.

\section{Local Representation Numbers, Exponential Sums, and Quadratic Gauss Sums} \label{section:LocalRepnNumsExpSumsGaussSums} 

We compute local densities by computing certain local representation numbers. In this section, we give some formulas for local representation numbers using Gauss sums and exponential sums. Throughout this paper, we abbreviate $e^{2 \pi i w}$ as $\e{w}$. For an integer $a$ and a positive integer $q$, the \emph{quadratic Gauss sum $\gausssum{a}{q}$ over $\ZZ / q \ZZ$} is defined by
\begin{align*}
\gausssum{a}{q} := \sum_{j \pmod{q}} \e{\frac{a j^2}{q}} = \sum_{j \in \ZZ / q\ZZ} \e{\frac{a j^2}{q}} = \sum_{j=0}^{q-1} \e{\frac{a j^2}{q}}.
\end{align*}
Unless otherwise specified, the term ``Gauss sum'' is taken to refer to a quadratic Gauss sum.

We begin this section by stating a formula for an exponential sum that comes up frequently in number theory. This formula follows from the fact that the sum is also a geometric sum.
\begin{lemma} \label{lem:sumelin}
Let $a, q \in \ZZ$ and $q>0$. Then $$ \sum_{t=0}^{q-1} \e{\frac{at}{q}} = \begin{cases} q, \;\text{if } a \equiv 0 \pmod{q}, \\ 0, \;\text{otherwise.} \end{cases}$$
\end{lemma}

Using the previous lemma, we prove the following theorem, which shows why Gauss sums are useful in computing local representation numbers.
\begin{theorem} \label{thm:rnQm}
Let $a$, $b$, and $c$ be integers, and let $n$ be a positive integer. Let $Q(\bvec{v}) = a x^2 + b y^2 + c z^2$. Then 
\begin{align*}
r_{n}(m , Q) &=  \frac{1}{n} \sum_{t=0}^{n - 1} \e{\frac{-m t}{n}} \gausssum{a t}{n} \gausssum{b t}{n} \gausssum{c t}{n} .
\end{align*}
\end{theorem}
\begin{proof}
By Lemma~\ref{lem:sumelin}, 
\begin{align*}
\frac{1}{n} \sum_{t=0}^{n - 1} \e{\frac{(Q(\bvec{v}) - m)t}{n}} 
=
\begin{cases}
1, \quad \text{if } Q(\bvec{v}) \equiv m \pmod{n}, \\
0, \quad \text{otherwise.}
\end{cases}
\end{align*}
Therefore,
\begin{align*}
r_{n}(m , Q) &= \sum_{\bvec{v} \in ( \ZZ / n \ZZ )^3} \frac{1}{n} \sum_{t=0}^{n - 1} \e{\frac{(Q(\bvec{v}) - m)t}{n}} \\
	&=  \sum_{x=0}^{n - 1} \sum_{y=0}^{n - 1} \sum_{z=0}^{n - 1} \frac{1}{n} \sum_{t=0}^{n - 1} \e{\frac{(a x^2 + b y^2 + c z^2 - m)t}{n}} \\
	&=  \frac{1}{n} \sum_{t=0}^{n - 1} \e{\frac{-m t}{n}} \sum_{x=0}^{n - 1} \e{\frac{a t x^2}{n}} \sum_{y=0}^{n - 1} \e{\frac{b t y^2}{n}} \sum_{z=0}^{n - 1} \e{\frac{c t z^2}{n}} \\
	&=  \frac{1}{n} \sum_{t=0}^{n - 1} \e{\frac{-m t}{n}} \gausssum{a t}{n} \gausssum{b t}{n} \gausssum{c t}{n} .
\end{align*}
\end{proof}

The next result is a special case of Theorem~\ref{thm:rnQm}, describing what happens when $n$ is a prime power.
\begin{corollary} \label{cor:rpkQm}
Let $a$, $b$, and $c$ be integers, let $k$ be a positive integer, and let $p$ be a prime. Let $Q(\bvec{v}) = a x^2 + b y^2 + c z^2$. Then 
\begin{align}
r_{p^k}(m , Q) &=  \frac{1}{p^k} \sum_{t=0}^{p^k - 1} \e{\frac{-m t}{p^k}} \gausssum{a t}{p^k} \gausssum{b t}{p^k} \gausssum{c t}{p^k} \label{eq:rpkQm} .
\end{align}
This may also be written as 
\begin{align}
r_{p^k}(m , Q) &=  p^{2k} + \frac{1}{p^k} \sum_{t=1}^{p^k - 1} \e{\frac{-m t}{p^k}} \gausssum{a t}{p^k} \gausssum{b t}{p^k} \gausssum{c t}{p^k} \label{eq:rpkQm 0+other} .
\end{align}
\end{corollary}

Knowing this, it would be useful to evaluate Gauss sums. 
Most formulas for a Gauss sum $\gausssum{a}{p^k}$ assume that $a$ is coprime to the prime $p$. 
In order to remove this assumption, we use the following lemma to transform sums of periodic functions into sums with fewer summands. The lemma follows from the periodicity of the function.
\begin{lemma} \label{lem:sumperiodic}
Let $f: \ZZ \to \CC$ be periodic modulo $n$. If $m$ is a positive integer, then
\begin{align*}
\sum_{t \in \ZZ/nm\ZZ} f(t) = \sum_{t=0}^{nm-1} f(t) = m \sum_{t \in \ZZ/n\ZZ} f(t) = m \sum_{t=0}^{n-1} f(t) .
\end{align*}
\end{lemma}

Using the previous lemma, we relate $\gausssum{a}{p^k}$ to another Gauss sum $\gausssum{a_0}{p^k}$, where $a_0$ is coprime to $p$.
\begin{lemma} \label{lem:sumGaussmult}
Suppose $k$ is a positive integer, $p$ is a prime, and $a \ne 0$. 
Let $a = a_0 p^\ell$ so that $\gcd(a_0, p) = 1$. If $\ell \le k$, then
\begin{align*}
\gausssum{a}{p^k} = p^{\ell} \gausssum{a_0}{p^{k-\ell}} .
\end{align*}
\end{lemma}
\begin{proof}
Using the definition of a quadratic Gauss sum and Lemma~\ref{lem:sumperiodic},
\begin{align*}
\gausssum{a}{p^k} &= \sum_{j=0}^{p^k  - 1} \e{\frac{a j^2}{p^k}} = \sum_{j=0}^{p^k  - 1} \e{\frac{a_0 p^\ell j^2}{p^k}} = \sum_{j=0}^{p^k  - 1} \e{\frac{a_0 j^2}{p^{k-\ell}}} \\
	&= p^{\ell} \sum_{j=0}^{p^{k-\ell}  - 1} \e{\frac{a_0 j^2}{p^{k-\ell}}} = p^{\ell} \gausssum{a_0}{p^{k-\ell}}.
\end{align*}
\end{proof}

The following lemma states a formula for a Gauss sum $\gausssum{a}{p^k}$ when $a$ is coprime to the prime $p$. The lemma is a special case of Theorem~1.5.2 in \cite[p.~26]{BEWGaussSums}.
\begin{lemma} \label{lem:gausssumppowsrelprime}
Suppose $k$ is a positive integer and $p$ is an odd prime.  Suppose $\gcd(a,p) = 1$. Then 
\begin{align*}
\gausssum{a}{p^k} = p^{k/2} \kronsym{a}{p^k} \varepsilon_{p^k} ,
\end{align*}
where $\kronsym{\cdot}{p^k}$ is the Jacobi symbol and 
\begin{align*}
\varepsilon_{p^k} =
\begin{cases}
1, &\text{if } p^k \equiv 1 \pmod{4}, \\
i, &\text{if } p^k \equiv 3 \pmod{4}.
\end{cases}
\end{align*}
\end{lemma}

Extending the previous lemma, the following lemma helps us evaluate a Gauss sum $\gausssum{a}{q}$ when $q$ is an odd prime power.

\begin{lemma} \label{lem:gausssumppows}
Suppose $k$ is a positive integer, $p$ is an odd prime, and $a \ne 0$. 
Let $a = a_0 p^\ell$ so that $\gcd(a_0, p) = 1$. Then
\begin{align*}
\gausssum{a}{p^k} = 
\begin{cases}
p^k , &\text{if } k \le \ell, \\
p^{(k+\ell)/2} \kronsym{a_0}{p^{k-\ell}} \varepsilon_{p^{k-\ell}} , &\text{if } k > \ell ,
\end{cases}
\end{align*}
where
$\varepsilon_{p^k}$ is as defined in Lemma~\ref{lem:gausssumppowsrelprime}.
\end{lemma}
\begin{proof}
Suppose $k \le \ell$. By the definition of a quadratic Gauss sum,
\begin{align*}
\gausssum{a}{p^k} = \sum_{j=0}^{p^k  - 1} \e{\frac{a j^2}{p^k}} = \sum_{j=0}^{p^k  - 1} \e{a_0 p^{\ell-k} j^2} = \sum_{j=0}^{p^k  - 1} 1 = p^k .
\end{align*}
Suppose $k > \ell$. 
By Lemma~\ref{lem:sumGaussmult}, $\gausssum{a}{p^k} = p^{\ell} \gausssum{a_0}{p^{k-\ell}}$.
We apply Lemma~\ref{lem:gausssumppowsrelprime} to see that
\begin{align*}
\gausssum{a}{p^k} = p^{\ell} p^{(k-\ell)/2} \kronsym{a_0}{p^{k-\ell}} \varepsilon_{p^{k-\ell}} = p^{(k+\ell)/2} \kronsym{a_0}{p^{k-\ell}} \varepsilon_{p^{k-\ell}} .
\end{align*}
\end{proof}

\begin{remark}
A simple calculation shows that $\varepsilon_{p^k}^2 = \kronsym{-1}{p^k} = \kronsym{-1}{p}^k$. Also, \mbox{$\varepsilon_{p^k}^4 = 1$.}
\end{remark}

We now provide a link between Gauss sums and exponential sums.
\begin{lemma} \label{lem:GaussKronId}
Suppose $p$ is an odd prime and $a \in \ZZ$.
 Then
\begin{align}
\gausssum{a}{p} = \sum_{t=0}^{p-1} \left( 1 + \kronsym{t}{p} \right) \e{\frac{at}{p}}. \label{eq:GaussKronId}
\end{align}

If $a \not\equiv 0 \pmod{p}$, then 
\begin{align*}
\gausssum{a}{p} = \sum_{t=0}^{p-1} \kronsym{t}{p} \e{\frac{at}{p}} = \sum_{t=1}^{p-1} \kronsym{t}{p} \e{\frac{at}{p}}. 
\end{align*}
\end{lemma}

\begin{proof}
Let $t$ be an integer. 
As noted by Cohen~\cite[p.~27]{CohenCompAlgNum}, the number of solutions modulo $p$ of the congruence
$$ j^2 \equiv t \pmod{p} $$
is $1+\kronsym{t}{p}$. Therefore,
\begin{align*}
\gausssum{a}{p} &= \sum_{j=0}^{p-1} \e{\frac{a j^2}{p}} = \sum_{t=0}^{p-1} \left( 1 + \kronsym{t}{p} \right) \e{\frac{at}{p}}.
\end{align*}

When $a \not\equiv 0 \pmod{p}$, we see that 
$$ \gausssum{a}{p} = \sum_{t=0}^{p-1} \kronsym{t}{p} \e{\frac{at}{p}} $$
follows from \eqref{eq:GaussKronId} and Lemma~\ref{lem:sumelin}.

The last equality in Lemma~\ref{lem:GaussKronId} follows from the fact that $\kronsym{0}{p} = 0$.
\end{proof}

\section{Some Supporting Lemmas} \label{section:SuppLems} 

This section contains many lemmas used in our computations for local densities at odd primes. 
This first lemma gives the value of a certain exponential sum and follows immediately from Lemma~\ref{lem:sumelin}.
\begin{lemma} \label{lem:sumelin 1toq}
Let $a, q \in \ZZ$ and $q>0$. If $a \not\equiv 0 \pmod{q}$, then 
\begin{align*}
\sum_{t=1}^{q-1} \e{\frac{at}{q}} = -1 .
\end{align*}
\end{lemma}

This next lemma says that the value of a Gauss sum can be determined by certain residue classes and valuations. For a prime $p$, we say $p^\ell \parallel a$ if $p^\ell \mid a$ but $p^{\ell+1} \nmid a$.
\begin{lemma} \label{lem:gausssumperiodic}
Suppose $k$ is a positive integer, $p$ is an odd prime, $a \ne 0$, and $b \ne 0$. Suppose $p^\ell \parallel a$ and $p^\ell \parallel b$ for some $\ell \ge 0$ so that there exist $a_0, b_0 \in \ZZ$ such that $\gcd(a_0 , p) = 1$, $\gcd(b_0 , p) = 1$, $a = a_0 p^\ell$, and $b = b_0 p^\ell$. If $a_0 \equiv b_0 \pmod{p}$, then
\begin{align*}
\gausssum{a}{p^k} = \gausssum{b}{p^k}.
\end{align*}
\end{lemma}
\begin{proof}
Since $a_0 \equiv b_0 \pmod{p}$, we have  
\begin{align*}
\kronsym{a_0}{p^{k-\ell}} = \kronsym{a_0}{p}^{k-\ell} = \kronsym{b_0}{p}^{k-\ell} = \kronsym{b_0}{p^{k-\ell}} .
\end{align*}
By Lemma~\ref{lem:gausssumppows}, we see that $\gausssum{a}{p^k} = \gausssum{b}{p^k}$.
\end{proof}

Now the sums of powers of Legendre symbols will come up repeatedly in our computations. This next lemma helps us compute them and follows from the fact that the Legendre symbol is a nontrivial real Dirichlet character modulo $p$.
\begin{lemma} \label{lem:sumLegendre}
Let $p$ be an odd prime, $k \ge 1$, and $m$ be an integer. Then
\begin{align*}
\sum_{t \in {(\ZZ/p^k\ZZ)}^*} \kronsym{t}{p}^m &=
\begin{cases}
p^k \left(1 - \dfrac{1}{p}\right), &\text{if $m$ is even,} \\
0, &\text{if $m$ is odd.}
\end{cases}
\end{align*}
\end{lemma}

Sums with Legendre symbols and exponentials occur frequently in calculations of local representation numbers. The following lemma helps us compute them.
\begin{lemma} \label{lem:sumeLegendre}
Let $p$ be an odd prime, $n$ be an integer, $k \ge 1$, and $m$ be an integer relatively prime to $p$, i.e., $\gcd(m,p) = 1$. Then
\begin{align*}
\sum_{t \in {(\ZZ/p^k\ZZ)}^*}  \e{\frac{m t}{p}} \kronsym{t}{p}^{n}&= \sum_{t \in \ZZ/p^k\ZZ}  \e{\frac{m t}{p}} \kronsym{t}{p}^{n} \\
&=
\begin{cases}
-p^{k-1}, &\text{if $n$ is even,} \\
p^{k-1/2} \kronsym{m}{p} \varepsilon_p , &\text{if $n$ is odd.}
\end{cases}
\end{align*}
\end{lemma}
\begin{proof}
Because $\kronsym{t}{p} = 0$ if $\gcd(t, p) > 1$, we have  
\begin{align*}
\sum_{t \in {(\ZZ/p^k\ZZ)}^*}  \e{\frac{m t}{p}} \kronsym{t}{p}^{n} = \sum_{t \in \ZZ/p^k\ZZ}  \e{\frac{m t}{p}} \kronsym{t}{p}^{n} .
\end{align*}
Since $f(t) = \e{\dfrac{m t}{p}} \kronsym{t}{p}^{n}$ is periodic modulo $p$, by Lemma~\ref{lem:sumperiodic}, 
\begin{align*}
\sum_{t \in \ZZ/p^k\ZZ}  \e{\frac{m t}{p}} \kronsym{t}{p}^{n} &= p^{k-1} \sum_{t=0}^{p-1} \e{\frac{m t}{p}} \kronsym{t}{p}^{n}.
\end{align*}

If $n$ is even, then 
\begin{align*}
p^{k-1} \sum_{t=0}^{p-1} \e{\frac{m t}{p}} \kronsym{t}{p}^{n} = p^{k-1} \sum_{t=1}^{p-1} \e{\frac{m t}{p}} = - p^{k-1} 
\end{align*}
by Lemma~\ref{lem:sumelin 1toq}.

If $n$ is odd, then 
\begin{align*}
p^{k-1} \sum_{t=0}^{p-1} \e{\frac{m t}{p}} \kronsym{t}{p}^{n} = p^{k-1} \sum_{t=0}^{p-1} \kronsym{t}{p} \e{\frac{m t}{p}} = p^{k-1} \gausssum{m}{p}
\end{align*}
by Lemma~\ref{lem:GaussKronId}. By Lemma~\ref{lem:gausssumppowsrelprime}, 
\begin{align*}
p^{k-1} \sum_{t=0}^{p-1} \e{\frac{m t}{p}} \kronsym{t}{p}^{n} = p^{k-1} p^{1/2} \kronsym{m}{p} \varepsilon_p = p^{k-1/2} \kronsym{m}{p} \varepsilon_p 
\end{align*}
if $n$ is odd.
\end{proof}

The following lemma gives the sum of certain powers of Legendre symbols that are useful in computing local densities.
\begin{lemma} \label{lem:sumLegendrepows}
Let $n_1$, $n_2$, and $r$ be integers such that $0 \le n_1 \le n_2 \le k$ and $\gcd(r,p) = 1$. Then
\begin{align*}
\sum_{\tau = k - n_2}^{k - n_1} \kronsym{r}{p}^{k-\tau} &= \frac{n_2 - n_1 + 1}{2} + \frac{(-1)^{n_1} + (-1)^{n_2}}{4} \\
	&\qquad+ \kronsym{r}{p} \left( \frac{n_2 - n_1 + 1}{2} - \frac{(-1)^{n_1} + (-1)^{n_2}}{4} \right) .
\end{align*}
\end{lemma}
\begin{proof}
Let $\tau_1 = k - \tau$. Then
\begin{align*}
\sum_{\tau = k - n_2}^{k - n_1} \kronsym{r}{p}^{k-\tau} &= \sum_{\tau_1 = n_1}^{n_2} \kronsym{r}{p}^{\tau_1} = \sum_{\substack{\tau_1 = n_1 \\ \tau_1 \text{ is even}}}^{n_2} 1 + \sum_{\substack{\tau_1 = n_1 \\ \tau_1 \text{ is odd}}}^{n_2} \kronsym{r}{p} .
\end{align*}
By a simple counting argument, one can see that the number of even integers in the set $\{ n \in \ZZ : n_1 \le n \le n_2 \}$ is $\frac{n_2 - n_1 + 1}{2} + \frac{(-1)^{n_1} + (-1)^{n_2}}{4}$. In a similar manner, we can see that the number of odd integers in the set $\{ n \in \ZZ : n_1 \le n \le n_2 \}$ is $\frac{n_2 - n_1 + 1}{2} - \frac{(-1)^{n_1} + (-1)^{n_2}}{4}$. 
From this, the lemma follows.
\end{proof}

\section{Proof of Theorem~\ref{thm:densities}} \label{section:Proof} 
Since we have stated some supporting lemmas in the previous section, we can now give a proof of Theorem~\ref{thm:densities} in this section. This section is split up into a proof of the formulas of $\alpha_p (m, Q)$ when $m \ne 0$ and a proof of the formulas for $\alpha_p (0, Q)$.
In this section, let $a$, $b$, $b_0$, $b_1$, $c$, $c_0$, $c_1$, and $p$ be the same as in the statement of Theorem~\ref{thm:densities}.

\subsection{Proof of Formulas for $\alpha_p (m, Q)$ for Nonzero $m$} 
This subsection proves the formulas of $\alpha_p (m, Q)$ when $m \ne 0$. In this subsection, let $m$, $m_0$, and $m_1$ be the same as in the statement of Theorem~\ref{thm:densities}. Suppose $k > m_1$ for this entire subsection and its lemmas. Towards computing $\alpha_p (m, Q)$, we compute $r_{p^k}(m , Q)$ and then take the appropriate limit.

By \eqref{eq:rpkQm 0+other},
\begin{align*}
r_{p^k}(m , Q) &=  p^{2k} + \frac{1}{p^k} \sum_{t=1}^{p^k - 1} \e{\frac{-m_0 p^{m_1} t}{p^k}} \gausssum{a t}{p^k} \gausssum{b_0 p^{b_1} t}{p^k} \gausssum{c_0 p^{c_1} t}{p^k} .
\end{align*}
Let $t = t_0 p^\tau$, where $0 \le \tau \le k-1$ and $t_0 \in {(\ZZ/p^{k-\tau}\ZZ)}^*$. Then
\begin{align}
r_{p^k}(m , Q) &=  p^{2k} + \frac{1}{p^k} \sum_{\tau=0}^{k-1} s_{k,\tau} , \label{eq:sum sktau}
\end{align}
where 
\begin{align*}
s_{k,\tau} &= \sum_{t_0 \in {(\ZZ/p^{k-\tau}\ZZ)}^*}  \e{\frac{-m_0 t_0 p^{m_1+\tau}}{p^k}} \gausssum{a t_0 p^\tau}{p^k} \gausssum{b_0 t_0 p^{b_1+\tau}}{p^k} \gausssum{c_0 t_0 p^{c_1+\tau}}{p^k} .
\end{align*}

Now the problem of computing $r_{p^k}(m, Q)$ becomes the problem of computing $s_{k,\tau}$ for $0 \le \tau \le k-1$. Our first lemma tells us that we only need to be concerned with $m_1 + 1$ values of $\tau$.

\begin{lemma} \label{lem:sktau tau0tok-m1-2}
If $0 \le \tau \le k - m_1 - 2$, then $s_{k,\tau} = 0$.
\end{lemma}
\begin{proof}
Suppose that $0 \le \tau \le k - m_1 - 2$. Then let $t_0 = t_1 + t_2 p$, where $1 \le t_1 \le p-1$ and $0 \le t_2 \le p^{k-\tau-1} -1$, so by Lemma~\ref{lem:gausssumperiodic},
\begin{align*}
s_{k,\tau} &= \sum_{t_1 = 1}^{p-1} \sum_{t_2 = 0}^{p^{k-\tau-1} -1} \e{\frac{-m_0 (t_1 + t_2 p)  p^{m_1+\tau}}{p^k}} \gausssum{a (t_1 + t_2 p)  p^\tau}{p^k} \\
		&\qquad\qquad\qquad \cdot \gausssum{b_0 (t_1 + t_2 p)  p^{b_1+\tau}}{p^k} \gausssum{c_0 (t_1 + t_2 p)  p^{c_1+\tau}}{p^k} \\
	&= \sum_{t_1 = 1}^{p-1} \sum_{t_2 = 0}^{p^{k-\tau-1} -1} \e{\frac{-m_0 t_1}{p^{k - m_1 - \tau}}} \e{\frac{-m_0 t_2}{p^{k - m_1 - 1 - \tau}}} \gausssum{a t_1 p^\tau}{p^k} \\
		&\qquad\qquad\qquad \cdot \gausssum{b_0 t_1  p^{b_1+\tau}}{p^k} \gausssum{c_0 t_1 p^{c_1+\tau}}{p^k} \\
	&= \sum_{t_1 = 1}^{p-1} \e{\frac{-m_0 t_1}{p^{k - m_1 - \tau}}} \gausssum{a t_1 p^\tau}{p^k}  \gausssum{b_0 t_1  p^{b_1+\tau}}{p^k} \\
		&\qquad\qquad\qquad \cdot \gausssum{c_0 t_1 p^{c_1+\tau}}{p^k} \sum_{t_2 = 0}^{p^{k-\tau-1} -1} \e{\frac{-m_0 t_2}{p^{k - m_1 - 1 - \tau}}}  . \numberthis \label{eq:sum sktau 0 to k-m1-2}
\end{align*}
By Lemmas~\ref{lem:sumperiodic} and \ref{lem:sumelin} and the condition $0 \le \tau \le k - m_1 - 2$,
\begin{align*}
\sum_{t_2 = 0}^{p^{k-\tau-1} -1}  \e{\frac{-m_0 t_2}{p^{k - m_1 - 1 - \tau}}} &= p^{m_1} \sum_{t_2 = 0}^{p^{k-m_1-\tau-1} -1}  \e{\frac{-m_0 t_2}{p^{k - m_1 - 1 - \tau}}} = p^{m_1}  \cdot 0 = 0.
\end{align*}
Thus, by \eqref{eq:sum sktau 0 to k-m1-2}, 
we obtain $s_{k,\tau} = 0$ for $0 \le \tau \le k - m_1 - 2$.
\end{proof}

In our next lemma, we evaluate $s_{k,\tau}$ when $\tau$ is greater than or equal to $k - \min(m_1, b_1)$.
\begin{lemma} \label{lem:sktau taumintok-1}
If $k - \min(m_1, b_1) \le \tau \le k-1$, then 
\begin{align*}
s_{k,\tau} = 
\begin{cases}
p^{3k + (k-\tau)/2} \left(1 - \dfrac{1}{p}\right), &\text{if $k - \tau$ is even,} \\
0, &\text{if $k - \tau$ is odd.}
\end{cases}
\end{align*}
\end{lemma}
\begin{proof}
Suppose that $k - \min(m_1, b_1) \le \tau \le k-1$. Then $p^k \mid -m_0 t_0 p^{m_1+\tau}$ for all $t_0 \in {(\ZZ/p^{k-\tau}\ZZ)}^*$ and $k \le b_1 + \tau \le c_1 + \tau$. Therefore, for all $t_0 \in {(\ZZ/p^{k-\tau}\ZZ)}^*$, we have $\e{\frac{-m_0 t_0 p^{m_1+\tau}}{p^k}} = 1$. By Lemma~\ref{lem:gausssumppows},
\begin{align*}
s_{k,\tau} &=\sum_{t_0 \in {(\ZZ/p^{k-\tau}\ZZ)}^*}  p^{(k+\tau)/2} \kronsym{a t_0}{p^{k-\tau}} \varepsilon_{p^{k-\tau}} p^{2k} \\
	&= \varepsilon_{p^{k-\tau}} p^{5 k/2 + \tau/2}\kronsym{a}{p}^{k-\tau}  \sum_{t_0 \in {(\ZZ/p^{k-\tau}\ZZ)}^*}  \kronsym{t_0}{p} ^{k-\tau}.
\end{align*}
By Lemma~\ref{lem:sumLegendre},
\begin{align*}
s_{k,\tau} &=
\begin{cases}
\varepsilon_{p^{k-\tau}} p^{5 k/2 + \tau/2}\kronsym{a}{p}^{k-\tau}  p^{k-\tau} \left(1 - \dfrac{1}{p}\right), &\text{if $k-\tau$ is even,} \\
0, &\text{if $k-\tau$ is odd.}
\end{cases}
\end{align*}
Because $\varepsilon_{p^{k-\tau}} = \kronsym{a}{p}^{k-\tau} = 1$ if $k-\tau$ is even, we have 
\begin{align*}
s_{k,\tau} &=
\begin{cases}
p^{3k + (k-\tau)/2}\left(1 - \dfrac{1}{p}\right), &\text{if $k-\tau$ is even,} \\
0, &\text{if $k-\tau$ is odd.}
\end{cases}
\end{align*}
\end{proof}

The following lemma gives the value of the sum $\sum_{\tau = k - \min(m_1, b_1)}^{k - 1}  s_{k,\tau}$.
\begin{lemma} \label{lem:sktausum tau-n1tok-1}
Let $n_1 = \min(m_1, b_1)$. 
Then
\begin{align} \label{eq:sktausum tau-n1tok-1}
\sum_{\tau = k - n_1}^{k - 1}  s_{k,\tau} = \sum_{\substack{\tau = k - n_1 \\ k-\tau \text{ is even}}}^{k - 1} p^{3k + (k-\tau)/2} \left(1 - \dfrac{1}{p}\right) = p^{3k}  ( p^{\lfloor n_1 / 2 \rfloor} -1)  ,
\end{align}
where $\lfloor x \rfloor$ is the greatest integer less than or equal to $x$.
\end{lemma}
\begin{proof}
The first equality in \eqref{eq:sktausum tau-n1tok-1} follows from Lemma~\ref{lem:sktau taumintok-1}.

It remains to prove the second equality in \eqref{eq:sktausum tau-n1tok-1}.
Let $\tau_1 = \frac{k-\tau}{2}$. Then
\begin{align*}
\sum_{\substack{\tau = k - n_1 \\ k-\tau \text{ is even}}}^{k - 1} p^{3k + (k-\tau)/2} \left(1 - \dfrac{1}{p}\right) &= p^{3k}  \left(1 - \dfrac{1}{p}\right) \sum_{\tau_1 = 1}^{\lfloor n_1 / 2 \rfloor} p^{\tau_1} .
\end{align*}
The lemma follows by noting that the latter sum is a geometric sum that equals
\begin{align*}
\dfrac{p ( p^{\lfloor n_1 / 2 \rfloor} -1)}{p - 1} .
\end{align*}
\end{proof}

\subsubsection{$\alpha_p (m, Q)$ when $m_1 < b_1$}
We start with computing formulas for $\alpha_p (m, Q)$ when $m_1 < b_1$ by computing $r_{p^k}(m, Q)$ when $m_1 < b_1$. 
The next lemma helps us compute $r_{p^k}(m, Q)$ when $m_1 < b_1$ by determining the value of $s_{k, k - m_1 - 1}$ when $m_1 < b_1$.
\begin{lemma} \label{lem:sktau m1<b1 tau=k-m1-1}
If $m_1 < b_1$, then
\begin{align*}
s_{k, k - m_1 - 1} = 
\begin{cases}
p^{3k + m_1/2} \kronsym{a m_0}{p}, &\text{if $m_1$ is even,} \\
-p^{3k + m_1/2 - 1/2}, &\text{if $m_1$ is odd.}
\end{cases}
\end{align*}
\end{lemma}
\begin{proof}
Since $m_1 < b_1$, we have $k \le b_1 + k - m_1 - 1 \le c_1 + k - m_1 - 1$. Thus, by Lemma~\ref{lem:gausssumppows}, 
\begin{align*}
s_{k,k - m_1 - 1} &=\sum_{t_0 \in {(\ZZ/p^{m_1 + 1}\ZZ)}^*}  \e{\frac{-m_0 t_0}{p}} p^{(2k - m_1 - 1)/2} \kronsym{a t_0}{p^{m_1 + 1}} \varepsilon_{p^{m_1 + 1}} p^{2k} \\
	&= p^{3k - m_1/2 - 1/2} \varepsilon_{p^{m_1 + 1}} \kronsym{a}{p}^{m_1 + 1} \sum_{t_0 \in {(\ZZ/p^{m_1 + 1}\ZZ)}^*}  \e{\frac{-m_0 t_0}{p}} \kronsym{t_0}{p}^{m_1 + 1} .
\end{align*}
By Lemma~\ref{lem:sumeLegendre},
\begin{align*}
s_{k,k - m_1 - 1} &=
\begin{cases}
 p^{3k - m_1/2 - 1/2} \varepsilon_{p} \kronsym{a}{p} p^{m_1+1/2} \kronsym{-m_0}{p} \varepsilon_p , &\text{if $m_1$ is even,} \\
- p^{3k - m_1/2 - 1/2} p^{m_1}, 	&\text{if $m_1$ is odd,}
\end{cases}\\
&=
\begin{cases}
 p^{3k + m_1/2} \kronsym{a m_0}{p} , &\text{if $m_1$ is even,} \\
- p^{3k + m_1/2 - 1/2} , 	&\text{if $m_1$ is odd,}
\end{cases}
\end{align*}
since $\varepsilon_p^2 = \kronsym{-1}{p}$.
\end{proof}

Using  Lemmas~\ref{lem:sktau tau0tok-m1-2}, \ref{lem:sktausum tau-n1tok-1}, and \ref{lem:sktau m1<b1 tau=k-m1-1}, we can compute $r_{p^k}(m, Q)$ when $m_1 < b_1$. By using the lemmas and \eqref{eq:sum sktau}, we see that if $m_1 < b_1$, then
\begin{align*}
r_{p^k}(m, Q) &=  p^{2k} + \frac{1}{p^k} s_{k, k - m_1 - 1} + \frac{1}{p^k} \sum_{\tau=k - m_1}^{k - 1} s_{k,\tau} \\
	&=
		\begin{cases}
		p^{2k} + p^{2k + m_1/2} \kronsym{a m_0}{p} + p^{2k} (p^{\lfloor m_1 /2 \rfloor} - 1), &\text{if $m_1$ is even,} \\
		p^{2k} - p^{2k + m_1/2 - 1/2} + p^{2k} (p^{\lfloor m_1 /2 \rfloor} - 1), &\text{if $m_1$ is odd,}
		\end{cases} \\
	&=
		\begin{cases}
		p^{2k} \left(1 + p^{m_1/2} \kronsym{a m_0}{p} + (p^{m_1 /2} - 1) \right), &\text{if $m_1$ is even,} \\
		p^{2k} \left(1 - p^{(m_1 - 1)/2} + (p^{(m_1 - 1)/2} - 1) \right), &\text{if $m_1$ is odd.}
		\end{cases} \\
	&=
		\begin{cases}
		p^{2k + m_1/2} \left( 1 + \kronsym{a m_0}{p} \right) , &\text{if $m_1$ is even,} \\
		0, &\text{if $m_1$ is odd.}
		\end{cases}
\end{align*}
By dividing by $p^{2k}$ and computing the limit as $k \to \infty$, we obtain \eqref{eq:density m1<b1} if $m_1 < b_1$.

\subsubsection{$\alpha_p (m, Q)$ when $b_1 \le m_1 < c_1$}
We start with computing formulas for $\alpha_p (m, Q)$ when $b_1 \le m_1 < c_1$ by computing $r_{p^k}(m, Q)$ when $b_1 \le m_1 < c_1$. 
The next lemmas help us compute  $r_{p^k}(m, Q)$ when $b_1 \le m_1 < c_1$. We first determine the value of $s_{k, k - m_1 - 1}$ when $b_1 \le m_1 < c_1$.
\begin{lemma} \label{lem:sktau b1<=m1<c1 tau=k-m1-1}
If $b_1 \le m_1 < c_1$, then
\begin{align*}
s_{k, k - m_1 - 1} = 
\begin{cases}
- p^{3k + b_1/2 - 1} \kronsym{-a b_0}{p}^{m_1 + 1} , &\text{if $b_1$ is even,}  \medskip\\ 
p^{3k + (b_1 - 1)/2} \kronsym{a}{p}^{m_1 + 1} \kronsym{b_0}{p}^{m_1} \kronsym{m_0}{p} , &\text{if $b_1$ is odd.}
\end{cases}
\end{align*}
\end{lemma}
\begin{proof}
Since $b_1 \le m_1 < c_1$, we have $b_1 + k - m_1 - 1 < k  \le c_1 + k - m_1 - 1$. Thus, by Lemma~\ref{lem:gausssumppows}, 
\begin{align*}
s_{k, k - m_1 - 1} &=\sum_{t_0 \in {(\ZZ/p^{m_1 + 1}\ZZ)}^*}  \e{\frac{-m_0 t_0}{p}} p^{(2k - m_1 - 1)/2} \kronsym{a t_0}{p^{m_1 + 1}} \varepsilon_{p^{m_1 + 1}} \\
	&\qquad\qquad\qquad\qquad \cdot p^{(2k + b_1 - m_1 - 1)/2} \kronsym{b_0 t_0}{p^{m_1 - b_1 + 1}} \varepsilon_{p^{m_1 - b_1 + 1}} p^k \\
	&= p^{3k + b_1/2 - m_1 - 1} \kronsym{a}{p}^{m_1 + 1} \kronsym{b_0}{p}^{m_1 - b_1 + 1} \\
	&\qquad\qquad \cdot \varepsilon_{p^{m_1 + 1}} \varepsilon_{p^{m_1 - b_1 + 1}} \sum_{t_0 \in {(\ZZ/p^{m_1 + 1}\ZZ)}^*}  \e{\frac{-m_0 t_0}{p}}  \kronsym{t_0}{p}^{- b_1} .
\end{align*}

Suppose $b_1$ is even. Then $\kronsym{b_0}{p}^{- b_1} = \kronsym{t_0}{p}^{- b_1} = 1$ and $\varepsilon_{p^{m_1 - b_1 + 1}} = \varepsilon_{p^{m_1 + 1}}$. Thus,
\begin{align*}
s_{k, k - m_1 - 1} &= p^{3k + b_1/2 - m_1 - 1} \kronsym{a}{p}^{m_1 + 1} \kronsym{b_0}{p}^{m_1 + 1} \varepsilon_{p^{m_1 + 1}}^2 \sum_{t_0 \in {(\ZZ/p^{m_1 + 1}\ZZ)}^*}  \e{\frac{-m_0 t_0}{p}} \\
	&= - p^{3k + b_1/2 - m_1 - 1} \kronsym{-a b_0}{p}^{m_1 + 1} p^{m_1} = - p^{3k + b_1/2 - 1} \kronsym{-a b_0}{p}^{m_1 + 1} 
\end{align*}
by Lemma~\ref{lem:sumeLegendre}.

Now suppose that $b_1$ is odd. Then $\varepsilon_{p^{m_1 + 1}} \varepsilon_{p^{m_1 - b_1 + 1}} = \varepsilon_p$, and 
\begin{align*}
s_{k, k - m_1 - 1} &= p^{3k + b_1/2 - m_1 - 1} \kronsym{a}{p}^{m_1 + 1} \kronsym{b_0}{p}^{m_1} \varepsilon_p \sum_{t_0 \in {(\ZZ/p^{m_1 + 1}\ZZ)}^*}  \e{\frac{-m_0 t_0}{p}}  \kronsym{t_0}{p} \\
	&= p^{3k + b_1/2 - m_1 - 1} \kronsym{a}{p}^{m_1 + 1} \kronsym{b_0}{p}^{m_1} \varepsilon_p^2 p^{m_1 + 1/2} \kronsym{-m_0}{p}
\end{align*}
by Lemma~\ref{lem:sumeLegendre}. Thus, if $b_1$ is odd, then 
\begin{align*}
s_{k, k - m_1 - 1} &= p^{3k + (b_1 - 1)/2} \kronsym{a}{p}^{m_1 + 1} \kronsym{b_0}{p}^{m_1} \kronsym{m_0}{p} .
\end{align*}
\end{proof}

The next lemma gives a formula for $s_{k, \tau}$ when $b_1 \le m_1 < c_1$ and $k - m_1 \le \tau < k - b_1$.
\begin{lemma} \label{lem:sktau b1<=m1<c1 tauk-m1tok-b1}
If $b_1 \le m_1 < c_1$ and $k - m_1 \le \tau < k - b_1$, then
\begin{align*}
s_{k, \tau} = 
\begin{cases}
p^{3k + b_1/2} \kronsym{- a b_0}{p}^{k - \tau} \left( 1 - \dfrac{1}{p} \right) &\text{if $b_1$ is even,} \\ \smallskip
0, &\text{if $b_1$ is odd.}
\end{cases}
\end{align*}
\end{lemma}
\begin{proof}
Since $b_1 \le m_1 < c_1$ and $k - m_1 \le \tau < k - b_1$, we have $\tau + b_1 < k \le \tau + m_1 < \tau + c_1$. Therefore, by Lemma~\ref{lem:gausssumppows},
\begin{align*}
s_{k,\tau} &=\sum_{t_0 \in {(\ZZ/p^{k-\tau}\ZZ)}^*}  p^{(k+\tau)/2} \kronsym{a t_0}{p^{k-\tau}} \varepsilon_{p^{k-\tau}} p^{(k + b_1 + \tau)/2} \kronsym{b_0 t_0}{p^{k - b_1 - \tau}} \varepsilon_{p^{k - b_1 - \tau}} p^k \\
	&=p^{2k + \tau + b_1/2} \varepsilon_{p^{k-\tau}} \varepsilon_{p^{k - b_1 - \tau}} \kronsym{a}{p}^{k-\tau} \kronsym{b_0}{p}^{k - b_1 - \tau} \sum_{t_0 \in {(\ZZ/p^{k-\tau}\ZZ)}^*}  \kronsym{t_0}{p}^{- b_1}
\end{align*}
By Lemma~\ref{lem:sumLegendre},
\begin{align*}
s_{k,\tau} &=
\begin{cases}
p^{2k + \tau + b_1/2} \varepsilon_{p^{k-\tau}}^2 \kronsym{a b_0}{p}^{k-\tau} p^{k-\tau} \left( 1 - \dfrac{1}{p} \right), &\text{if $b_1$ is even,} \\
0, &\text{if $b_1$ is odd,}
\end{cases}\\
	&=
\begin{cases}
p^{3k  + b_1/2} \kronsym{- a b_0}{p}^{k-\tau} \left( 1 - \dfrac{1}{p} \right), &\text{if $b_1$ is even,} \\
0, &\text{if $b_1$ is odd.}
\end{cases}
\end{align*}
\end{proof}

Using Lemmas~\ref{lem:sktau tau0tok-m1-2}, \ref{lem:sktausum tau-n1tok-1}, \ref{lem:sktau b1<=m1<c1 tau=k-m1-1}, \ref{lem:sktau b1<=m1<c1 tauk-m1tok-b1}, and \ref{lem:sumLegendrepows}, we can compute $r_{p^k}(m, Q)$ when $b_1 \le m_1 < c_1$. By using the lemmas and \eqref{eq:sum sktau}, we see that if $b_1 \le m_1 < c_1$, then 
\begin{align*}
r_{p^k}(m, Q) &=  p^{2k} + \frac{1}{p^k} s_{k, k - m_1 - 1} + \frac{1}{p^k} \sum_{\tau=k - m_1}^{k - b_1 - 1} s_{k,\tau} + \frac{1}{p^k} \sum_{\tau=k - b_1}^{k - 1} s_{k,\tau} \\
	&=  p^{2k} + \frac{1}{p^k} s_{k, k - m_1 - 1} + \frac{1}{p^k} \sum_{\tau=k - m_1}^{k - b_1 - 1} s_{k,\tau} + p^{2k} (p^{\lfloor b_1 / 2 \rfloor} - 1) \\
	&=
\begin{cases}
p^{2k} - p^{2k + b_1/2 - 1} \kronsym{-a b_0}{p}^{m_1 + 1} \\
	\qquad+ \displaystyle{\sum_{\tau=k - m_1}^{k - b_1 - 1}} p^{2k + b_1/2} \kronsym{- a b_0}{p}^{k - \tau} \left( 1 - \dfrac{1}{p} \right) + p^{2k} (p^{b_1 / 2} - 1), \\
	\qquad\qquad\qquad\qquad\qquad\qquad\qquad\qquad\qquad\qquad\text{if $b_1$ is even,} \\
p^{2k} + p^{2k + (b_1 - 1)/2} \kronsym{a}{p}^{m_1 + 1} \kronsym{b_0}{p}^{m_1} \kronsym{m_0}{p} + p^{2k} (p^{(b_1 - 1) / 2} - 1), \\
	\qquad\qquad\qquad\qquad\qquad\qquad\qquad\qquad\qquad\qquad\text{if $b_1$ is odd.}
\end{cases}\\
	&=
\begin{cases}
p^{2k + b_1 / 2} \left(1 - \dfrac{1}{p} \kronsym{-a b_0}{p}^{m_1 + 1} + \left( 1 - \dfrac{1}{p} \right) \left( \dfrac{m_1 - b_1}{2} \right.\right. \\
	\qquad\left.\left. + \dfrac{(-1)^{m_1} -1}{4} + \kronsym{- a b_0}{p} \left( \dfrac{m_1 - b_1}{2} + \dfrac{1 - (-1)^{m_1}}{4} \right) \right) \right), \\
	\qquad\qquad\qquad\qquad\qquad\qquad\qquad\qquad\qquad\qquad\qquad\text{if $b_1$ is even,} \\
p^{2k + (b_1 - 1) / 2} \left(1 + \kronsym{a}{p}^{m_1 + 1} \kronsym{b_0}{p}^{m_1} \kronsym{m_0}{p} \right), \qquad\text{if $b_1$ is odd.}
\end{cases}
\end{align*}
By dividing by $p^{2k}$ and computing the limit as $k \to \infty$, we obtain \eqref{eq:density b1<=m1<c1} if $b_1 \le m_1 < c_1$.

\subsubsection{$\alpha_p (m, Q)$ when $m_1 \ge c_1$}
We start with computing formulas for $\alpha_p (m, Q)$ when $m_1 \ge c_1$ by computing $r_{p^k}(m, Q)$ when $m_1 \ge c_1$. 
The next lemmas help us compute  $r_{p^k}(m, Q)$ when $m_1 \ge c_1$. We begin with computing $s_{k, k - m_1 - 1}$ when $m_1 \ge c_1$.
\begin{lemma} \label{lem:sktau m1>=c1 tau=k-m1-1}
Suppose $m_1 \ge c_1$.

If $b_1 \equiv c_1 \pmod{2}$, then 
\begin{align*}
s_{k, k - m_1 - 1} = 
\begin{cases}
p^{3k - m_1/2 + (b_1 + c_1)/2 - 1} \kronsym{a m_0}{p} \kronsym{- b_0 c_0}{p}^{b_1 + 1}, &\text{if $m_1$ is even,} \\
- p^{3k - (m_1 + 1)/2 + (b_1 + c_1)/2 - 1} \kronsym{- b_0 c_0}{p}^{b_1} , &\text{if $m_1$ is odd.}
\end{cases}
\end{align*}

If $b_1 \not\equiv c_1 \pmod{2}$, then 
\begin{align*}
s_{k, k - m_1 - 1} = 
\begin{cases}
- p^{3k - m_1/2 + (b_1 + c_1 - 1)/2 - 1} \kronsym{- a}{p} \kronsym{b_0}{p}^{b_1 + 1} \kronsym{c_0}{p}^{b_1} , &\text{if $m_1$ is even,} \smallskip\\
p^{3k - (m_1 + 1)/2 + (b_1 + c_1 - 1)/2} \kronsym{b_0}{p}^{b_1} \kronsym{c_0}{p}^{b_1 + 1} \kronsym{m_0}{p} , &\text{if $m_1$ is odd.}
\end{cases}
\end{align*}
\end{lemma}
\begin{proof}
Since $m_1 \ge c_1$, we have $b_1 + k - m_1 - 1 \le c_1 + k - m_1 - 1 \le k - 1 < k$. Thus, by Lemma~\ref{lem:gausssumppows}, 
\begin{align*}
s_{k, k - m_1 - 1} &=\sum_{t_0 \in {(\ZZ/p^{m_1 + 1}\ZZ)}^*}  \e{\frac{-m_0 t_0}{p}} p^{(2k - m_1 - 1)/2}  \kronsym{a t_0}{p^{m_1 + 1}} \varepsilon_{p^{m_1 + 1}} p^{(2k + b_1 - m_1 - 1)/2} \\
		&\qquad\qquad \cdot \kronsym{b_0 t_0}{p^{m_1 - b_1 + 1}} \varepsilon_{p^{m_1 - b_1 + 1}} p^{(2k + c_1 - m_1 - 1)/2} \kronsym{c_0 t_0}{p^{m_1 - c_1 + 1}} \varepsilon_{p^{m_1 - c_1 + 1}} \\
	&=  p^{3k - 3 m_1/2 - 3/2 + (b_1 + c_1)/2} \varepsilon_{p^{m_1 + 1}} \varepsilon_{p^{m_1 - b_1 + 1}} \varepsilon_{p^{m_1 - c_1 + 1}} \kronsym{a}{p}^{m_1 + 1}  \kronsym{b_0}{p}^{m_1 - b_1 + 1} \\
		&\qquad\qquad \cdot \kronsym{c_0}{p}^{m_1 - c_1 + 1} \sum_{t_0 \in {(\ZZ/p^{m_1 + 1}\ZZ)}^*} \e{\frac{-m_0 t_0}{p}} \kronsym{t_0}{p}^{m_1 - b_1 - c_1 + 1} .
\end{align*}

Suppose that $b_1 \equiv c_1 \pmod{2}$. Then 
\begin{align*}
s_{k, k - m_1 - 1} &=  p^{3k - 3 m_1/2 - 3/2 + (b_1 + c_1)/2} \varepsilon_{p^{m_1 + 1}} \varepsilon_{p^{m_1 - b_1 + 1}}^2 \kronsym{a}{p}^{m_1 + 1}  \kronsym{b_0}{p}^{m_1 - b_1 + 1} \\
		&\qquad\qquad \cdot \kronsym{c_0}{p}^{m_1 - b_1 + 1} \sum_{t_0 \in {(\ZZ/p^{m_1 + 1}\ZZ)}^*} \e{\frac{-m_0 t_0}{p}} \kronsym{t_0}{p}^{m_1 + 1} \\
	&=  p^{3k - 3 m_1/2 - 3/2 + (b_1 + c_1)/2} \varepsilon_{p^{m_1 + 1}} \kronsym{a}{p}^{m_1 + 1}  \kronsym{- b_0 c_0}{p}^{m_1 - b_1 + 1} \\
		&\qquad\qquad \cdot \sum_{t_0 \in {(\ZZ/p^{m_1 + 1}\ZZ)}^*} \e{\frac{-m_0 t_0}{p}} \kronsym{t_0}{p}^{m_1 + 1} .
\end{align*}
By Lemma~\ref{lem:sumeLegendre},
\begin{align*}
s_{k, k - m_1 - 1} &=  
\begin{cases}
p^{3k - 3 m_1/2 - 3/2 + (b_1 + c_1)/2} \varepsilon_{p} \kronsym{a}{p} \kronsym{- b_0 c_0}{p}^{b_1 + 1} p^{m_1 + 1/2} \kronsym{-m_0}{p} \varepsilon_{p} ,	\\
	\qquad\qquad\qquad\qquad\qquad\qquad\qquad\qquad\qquad\qquad\text{if $m_1$ is even,} \smallskip\\
- p^{3k - 3 m_1/2 - 3/2 + (b_1 + c_1)/2} \kronsym{- b_0 c_0}{p}^{b_1} p^{m_1}  ,	\qquad\text{if $m_1$ is odd,}
\end{cases}
\end{align*}
which simplifies to what is stated in Lemma~\ref{lem:sktau m1>=c1 tau=k-m1-1} for $b_1 \equiv c_1 \pmod{2}$.

Now suppose $b_1 \not\equiv c_1 \pmod{2}$. Then $\varepsilon_{p^{m_1 + 1}} \varepsilon_{p^{m_1 - b_1 + 1}} \varepsilon_{p^{m_1 - c_1 + 1}} = \varepsilon_{p^{m_1 + 1}}^2 \varepsilon_{p^{m_1}} = \kronsym{-1}{p}^{m_1 + 1}\varepsilon_{p^{m_1}}$ and 
\begin{align*}
s_{k, k - m_1 - 1} &= p^{3k - 3 m_1/2 - 3/2 + (b_1 + c_1)/2}  \kronsym{-1}{p}^{m_1 + 1}\varepsilon_{p^{m_1}} \kronsym{a}{p}^{m_1 + 1}  \kronsym{b_0}{p}^{m_1 - b_1 + 1} \\
		&\qquad\qquad \cdot \kronsym{c_0}{p}^{m_1 - b_1} \sum_{t_0 \in {(\ZZ/p^{m_1 + 1}\ZZ)}^*} \e{\frac{-m_0 t_0}{p}} \kronsym{t_0}{p}^{m_1} \\
	&= p^{3k - 3 m_1/2 + (b_1 + c_1 - 1)/2 - 1}  \varepsilon_{p^{m_1}} \kronsym{- a}{p}^{m_1 + 1}  \kronsym{b_0}{p}^{m_1 + b_1 + 1} \\
		&\qquad\qquad \cdot \kronsym{c_0}{p}^{m_1 + b_1} \sum_{t_0 \in {(\ZZ/p^{m_1 + 1}\ZZ)}^*} \e{\frac{-m_0 t_0}{p}} \kronsym{t_0}{p}^{m_1} .
\end{align*}
By Lemma~\ref{lem:sumeLegendre},
\begin{align*}
s_{k, k - m_1 - 1} &=  
\begin{cases}
- p^{3k - 3 m_1/2 + (b_1 + c_1 - 1)/2 - 1}  \kronsym{- a}{p}  \kronsym{b_0}{p}^{b_1 + 1} \kronsym{c_0}{p}^{b_1} p^{m_1} ,	\\
	\qquad\qquad\qquad\qquad\qquad\qquad\qquad\qquad\qquad\qquad\text{if $m_1$ is even,} \smallskip\\
p^{3k - 3 m_1/2 + (b_1 + c_1 - 1)/2 - 1}  \varepsilon_{p} \kronsym{b_0}{p}^{b_1} \kronsym{c_0}{p}^{b_1 + 1} p^{m_1 + 1/2} \kronsym{-m_0}{p} \varepsilon_{p}  ,	\\
	\qquad\qquad\qquad\qquad\qquad\qquad\qquad\qquad\qquad\qquad\text{if $m_1$ is odd,}
\end{cases}
\end{align*}
which simplifies to what is stated in Lemma~\ref{lem:sktau m1>=c1 tau=k-m1-1} for $b_1 \not\equiv c_1 \pmod{2}$.
\end{proof}

The next lemma gives the value of $s_{k, \tau}$ when $m_1 \ge c_1$ and $k - m_1 \le \tau < k - c_1$.
\begin{lemma} \label{lem:sktau m1>=c1 tauk-m1tok-c1}
If $m_1 \ge c_1$ and $k - m_1 \le \tau < k - c_1$, then
\begin{align*}
s_{k, \tau} &= 
\begin{cases}
p^{3k - (k - \tau)/2 + (b_1 + c_1)/2} \kronsym{- b_0 c_0}{p}^{b_1} \left( 1 - \dfrac{1}{p} \right) ,	\\
	\qquad\qquad\qquad\qquad\qquad\qquad\qquad\text{if $b_1 \equiv c_1 \pmod{2}$ and $k - \tau$ is even,} \smallskip\\
p^{3k + (b_1 + c_1 - (k - \tau))/2} \kronsym{- a}{p} \kronsym{b_0}{p}^{b_1 + 1} \kronsym{c_0}{p}^{b_1} \left( 1 - \dfrac{1}{p} \right) ,	\\
	\qquad\qquad\qquad\qquad\qquad\qquad\qquad\text{if $b_1 \not\equiv c_1 \pmod{2}$ and $k - \tau$ is odd,} \\
0, \qquad\qquad\qquad\qquad\qquad\qquad\quad\text{otherwise.}
\end{cases}
\end{align*}
\end{lemma}
\begin{proof}
Since $m_1 \ge c_1$ and $k - m_1 \le \tau < k - c_1$, we have $\tau + b_1 \le \tau + c_1 < k \le m_1 + \tau$. Therefore, by Lemma~\ref{lem:gausssumppows},
\begin{align*}
s_{k,\tau} &= \sum_{t_0 \in {(\ZZ/p^{k-\tau}\ZZ)}^*}  p^{(k+\tau)/2} \kronsym{a t_0}{p^{k-\tau}} \varepsilon_{p^{k-\tau}} p^{(k + b_1 + \tau)/2} \kronsym{b_0 t_0}{p^{k - b_1 - \tau}} \varepsilon_{p^{k - b_1 - \tau}} \\
	&\qquad\qquad\qquad\qquad \cdot p^{(k + c_1 + \tau)/2} \kronsym{c_0 t_0}{p^{k - c_1 - \tau}} \varepsilon_{p^{k - c_1 - \tau}} \\
	&= p^{(3k + 3\tau + b_1 + c_1)/2} \varepsilon_{p^{k-\tau}} \varepsilon_{p^{k - b_1 - \tau}} \varepsilon_{p^{k - c_1 - \tau}} \kronsym{a}{p}^{k-\tau} \kronsym{b_0}{p}^{k - b_1 - \tau} \\
	&\qquad\qquad\qquad\qquad \cdot \kronsym{c_0}{p}^{k - c_1 - \tau} \sum_{t_0 \in {(\ZZ/p^{k-\tau}\ZZ)}^*}  \kronsym{t_0}{p}^{k-\tau - b_1 - c_1} .
\end{align*}

If $b_1 \equiv c_1 \pmod{2}$, then, by Lemma~\ref{lem:sumLegendre},
\begin{align*}
s_{k,\tau} &=
\begin{cases}
p^{(3k + 3\tau + b_1 + c_1)/2} \varepsilon_{p^{b_1}}^2 \kronsym{b_0}{p}^{b_1} \kronsym{c_0}{p}^{b_1} p^{k-\tau} \left( 1 - \dfrac{1}{p} \right) , &\text{if $k-\tau$ is even,} \\
0 , &\text{if $k-\tau$ is odd,} 
\end{cases} \\
	&=
\begin{cases}
p^{3k - (k - \tau)/2 + (b_1 + c_1)/2} \kronsym{- b_0 c_0}{p}^{b_1} \left( 1 - \dfrac{1}{p} \right) , &\text{if $k-\tau$ is even,} \\
0 , &\text{if $k-\tau$ is odd.} 
\end{cases}
\end{align*}

If $b_1 \not\equiv c_1 \pmod{2}$, then, by Lemma~\ref{lem:sumLegendre},
\begin{align*}
s_{k,\tau} &=
\begin{cases}
0 , &\text{if $k-\tau$ is even,} \\
p^{(3k + 3\tau + b_1 + c_1)/2} \varepsilon_{p}^2 \kronsym{a}{p} \kronsym{b_0}{p}^{b_1 + 1} \kronsym{c_0}{p}^{b_1} p^{k-\tau} \left( 1 - \dfrac{1}{p} \right) , &\text{if $k-\tau$ is odd,} 
\end{cases} \\
	&=
\begin{cases}
0 , &\text{if $k-\tau$ is even,} \\
p^{3k + (b_1 + c_1 - (k - \tau))/2} \kronsym{- a}{p} \kronsym{b_0}{p}^{b_1 + 1} \kronsym{c_0}{p}^{b_1} \left( 1 - \dfrac{1}{p} \right) , &\text{if $k-\tau$ is odd.} 
\end{cases}
\end{align*}
\end{proof}

We would like to compute the sum $\sum_{\tau = k - m_1}^{k - c_1 - 1}  s_{k,\tau}$ when $m_1 \ge c_1$. In order to do that, we make use of some geometric-like sums given in the next lemma. We use the notation $\lceil x \rceil$ for the least integer greater than or equal to $x$. The details of the proof are omitted.
\begin{lemma} \label{lem:sum 1/p alt pows}
Let $n_1$ and $n_2$ be integers such that $0 \le n_1 \le n_2 \le k$. Then
\begin{align*}
\sum_{\substack{\tau = k - n_2 \\ k - \tau \text{ is even}}}^{k - n_1} p^{-(k - \tau)/2} \left( 1 - \dfrac{1}{p} \right) = p^{- \lceil n_1 / 2 \rceil} \left( 1 - p^{-\lfloor n_2 / 2 \rfloor + \lceil n_1 / 2 \rceil - 1} \right)
\end{align*}
and
\begin{align*}
\sum_{\substack{\tau = k - n_2 \\ k - \tau \text{ is odd}}}^{k - n_1} p^{-(k - \tau)/2} \left( 1 - \dfrac{1}{p} \right) &= p^{1/2 - \lceil (n_1 + 1)/ 2 \rceil} \left( 1 - p^{- \lfloor (n_2 + 1) / 2 \rfloor + \lceil (n_1 + 1) / 2 \rceil - 1} \right) .
\end{align*}
\end{lemma}

Using previous lemmas, we now compute the sum $\sum_{\tau = k - m_1}^{k - c_1 - 1}  s_{k,\tau}$ when $m_1 \ge c_1$.
\begin{lemma} \label{lem:sktausum m1>=c1 tauk-m1tok-c1}
If $m_1 \ge c_1$, then 
\begin{align*}
\sum_{\tau = k - m_1}^{k - c_1 - 1}  s_{k,\tau} &=
\begin{cases}
p^{3k + (b_1 + c_1)/2 - \lceil (c_1 + 1)/2 \rceil} \kronsym{- b_0 c_0}{p}^{b_1} \left( 1 - p^{- \lfloor m_1 / 2\rfloor + \lceil (c_1 + 1)/2 \rceil - 1} \right) , \\
	\qquad\qquad\qquad\qquad\qquad\qquad\qquad\qquad\qquad\qquad\text{if $b_1 \equiv c_1 \pmod{2}$,} \\
p^{3k + (b_1 + c_1 - 1)/2 - \lceil c_1/2 \rceil} \kronsym{-a}{p} \kronsym{b_0}{p}^{b_1 + 1} \kronsym{c_0}{p}^{b_1} \smallskip \\
	\qquad\qquad\qquad\qquad\qquad\qquad\qquad\qquad \cdot \left( 1- p^{- \lfloor (m_1 + 1) / 2\rfloor + \lceil c_1 /2 \rceil} \right) , \\
	\qquad\qquad\qquad\qquad\qquad\qquad\qquad\qquad\qquad\qquad\text{if $b_1 \not\equiv c_1 \pmod{2}$.} 
\end{cases}
\end{align*}
\end{lemma}
\begin{proof}
By Lemmas~\ref{lem:sktau m1>=c1 tauk-m1tok-c1} and \ref{lem:sum 1/p alt pows}, if $b_1 \equiv c_1 \pmod{2}$, then 
\begin{align*}
\sum_{\tau = k - m_1}^{k - c_1 - 1}  s_{k,\tau} &= p^{3k + (b_1 + c_1)/2} \kronsym{- b_0 c_0}{p}^{b_1} \sum_{\substack{\tau = k - m_1 \\ k - \tau \text{ is even}}}^{k - c_1 - 1}  p^{- (k - \tau)/2} \left( 1 - \dfrac{1}{p} \right) \\
	&= p^{3k + (b_1 + c_1)/2} \kronsym{- b_0 c_0}{p}^{b_1} p^{- \lceil (c_1 + 1)/2 \rceil} \left( 1 - p^{- \lfloor m_1 / 2\rfloor + \lceil (c_1 + 1)/2 \rceil - 1} \right) ,
\end{align*}
which simplifies to what is stated in Lemma~\ref{lem:sktausum m1>=c1 tauk-m1tok-c1} for $b_1 \equiv c_1 \pmod{2}$.

By Lemmas~\ref{lem:sktau m1>=c1 tauk-m1tok-c1} and \ref{lem:sum 1/p alt pows}, if $b_1 \not\equiv c_1 \pmod{2}$, then 
\begin{align*}
&\sum_{\tau = k - m_1}^{k - c_1 - 1}  s_{k,\tau} = p^{3k + (b_1 + c_1)/2} \kronsym{- a}{p} \kronsym{b_0}{p}^{b_1 + 1} \kronsym{c_0}{p}^{b_1} \sum_{\substack{\tau = k - m_1 \\ k - \tau \text{ is odd}}}^{k - c_1 - 1}  p^{- (k - \tau)/2} \left( 1 - \dfrac{1}{p} \right) \\
	&= p^{3k + (b_1 + c_1 + 1)/2} \kronsym{-a}{p} \kronsym{b_0}{p}^{b_1 + 1} \kronsym{c_0}{p}^{b_1} p^{- \lceil c_1/2 \rceil - 1}  \left( 1 - p^{- \lfloor (m_1 + 1) / 2\rfloor + \lceil c_1 /2 \rceil} \right) ,
\end{align*}
which simplifies to what is stated in Lemma~\ref{lem:sktausum m1>=c1 tauk-m1tok-c1} for $b_1 \not\equiv c_1 \pmod{2}$.
\end{proof}

We now compute $s_{k, \tau}$ when $m_1 \ge c_1$ and $k - c_1 \le \tau < k - b_1$.
\begin{lemma} \label{lem:sktau m1>=c1 tauk-c1tok-b1}
If $m_1 \ge c_1$ and $k - c_1 \le \tau < k - b_1$, then
\begin{align*}
s_{k, \tau} &= 
\begin{cases}
p^{3k + b_1/2} \kronsym{- a b_0}{p}^{k - \tau} \left( 1 - \dfrac{1}{p} \right) ,	&\text{if $b_1$ is even,} \smallskip\\
0 ,	&\text{if $b_1 $ is odd} 
\end{cases}
\end{align*}
\end{lemma}
\begin{proof}
Since $m_1 \ge c_1$ and $k - c_1 \le \tau < k - b_1$, we have $b_1 + \tau < k \le c_1 + \tau \le m_1 + \tau$. Therefore, by Lemma~\ref{lem:gausssumppows},
\begin{align*}
s_{k,\tau} &=\sum_{t_0 \in {(\ZZ/p^{k-\tau}\ZZ)}^*}  p^{(k+\tau)/2} \kronsym{a t_0}{p^{k-\tau}} \varepsilon_{p^{k-\tau}} p^{(k + b_1 + \tau)/2} \kronsym{b_0 t_0}{p^{k - b_1 - \tau}} \varepsilon_{p^{k - b_1 - \tau}} p^k .
\end{align*}
By the same methods used in the proof of Lemma~\ref{lem:sktau b1<=m1<c1 tauk-m1tok-b1}, we have the result of this lemma.
\end{proof}

The above lemma together with Lemma~\ref{lem:sumLegendrepows} now immediately gives the following.
\begin{lemma} \label{lem:sktausum m1>=c1 tauk-c1tok-b1}
If $m_1 \ge c_1$, then
\begin{align*}
\sum_{\tau = k - c_1}^{k - b_1 - 1} s_{k, \tau} &= 
\begin{cases}
p^{3k + b_1/2}\left( 1 - \dfrac{1}{p} \right) \left( \dfrac{c_1 - b_1}{2} + \dfrac{(-1)^{c_1} -1}{4} \right.\\
	\qquad\qquad\qquad\qquad\left.+ \kronsym{- a b_0}{p} \left( \dfrac{c_1 - b_1}{2} + \dfrac{1 - (-1)^{c_1}}{4} \right) \right) , \\
	\quad\qquad\text{if $b_1$ is even,} \smallskip\\
0 ,	\qquad\text{if $b_1 $ is odd.} 
\end{cases}
\end{align*}
\end{lemma}

Using Lemmas~\ref{lem:sktau tau0tok-m1-2}, \ref{lem:sktausum tau-n1tok-1}, \ref{lem:sktau m1>=c1 tau=k-m1-1}, \ref{lem:sktausum m1>=c1 tauk-m1tok-c1}, and \ref{lem:sktausum m1>=c1 tauk-c1tok-b1}, we now compute $r_{p^k}(m, Q)$ when $m_1 \ge c_1$. By using the Lemmas~\ref{lem:sktau tau0tok-m1-2} and \ref{lem:sktausum tau-n1tok-1} and \eqref{eq:sum sktau}, we see that
\begin{align*}
r_{p^k}(m, Q) &=  p^{2k} + \frac{1}{p^k} s_{k, k - m_1 - 1} + \frac{1}{p^k} \sum_{\tau=k - m_1}^{k - c_1 - 1} s_{k,\tau} + \frac{1}{p^k} \sum_{\tau=k - c_1}^{k - b_1 - 1} s_{k,\tau} + \frac{1}{p^k} \sum_{\tau=k - b_1}^{k - 1} s_{k,\tau} \\
	&=  p^{2k} + \frac{1}{p^k} s_{k, k - m_1 - 1} + \frac{1}{p^k} \sum_{\tau=k - m_1}^{k - c_1 - 1} s_{k,\tau} + \frac{1}{p^k} \sum_{\tau=k - c_1}^{k - b_1 - 1} s_{k,\tau} + p^{2k} (p^{\lfloor b_1 / 2 \rfloor} - 1) \\
	&=  \frac{1}{p^k} s_{k, k - m_1 - 1} + \frac{1}{p^k} \sum_{\tau=k - m_1}^{k - c_1 - 1} s_{k,\tau} + \frac{1}{p^k} \sum_{\tau=k - c_1}^{k - b_1 - 1} s_{k,\tau} + p^{2k + \lfloor b_1 / 2 \rfloor}.
\end{align*}
Substituting the identities from Lemmas~\ref{lem:sktau m1>=c1 tau=k-m1-1}, \ref{lem:sktausum m1>=c1 tauk-m1tok-c1}, and \ref{lem:sktausum m1>=c1 tauk-c1tok-b1} and simplifying, we deduce that
\begin{align*}
r_{p^k}(m, Q) &= 
\begin{cases}
p^{2k + b_1 / 2} \left( 1 + \dfrac{1}{p} + p^{- m_1/2 + c_1/2 - 1} \left( \kronsym{- a b_0 c_0 m_0}{p} - 1 \right) \right. \smallskip\\
	\qquad\qquad\qquad\left.+ \left( 1 - \dfrac{1}{p} \right) \left( \dfrac{c_1 - b_1}{2} + \kronsym{- a b_0}{p} \dfrac{c_1 - b_1}{2} \right) \right) , \\
	\qquad\qquad\qquad\qquad\qquad\qquad\qquad\qquad\text{if $c_1$ and $m_1$ are even,} \\
p^{2k + b_1 / 2} \left( \left( 1 + \dfrac{1}{p} \right) \left( 1 - p^{- (m_1 + 1)/2 + c_1/2}\right) \right. \smallskip\\
	\qquad\qquad\qquad\left.+ \left( 1 - \dfrac{1}{p} \right) \left( \dfrac{c_1 - b_1}{2} + \kronsym{- a b_0}{p} \dfrac{c_1 - b_1}{2} \right) \right), \\
	\qquad\qquad\qquad\qquad\qquad\qquad\qquad\qquad\text{if $c_1$ is even and $m_1$ is odd,} \\
p^{2k + b_1 / 2} \left( 1 - p^{- m_1/2 + (c_1 - 1)/2} \kronsym{- a b_0}{p} \left( 1 + \dfrac{1}{p} \right) + \dfrac{1}{p} \kronsym{-a b_0}{p} \right. \smallskip\\
	\qquad\qquad\qquad\left.+ \left( 1 - \dfrac{1}{p} \right) \left( \dfrac{c_1 - b_1 - 1}{2} + \kronsym{- a b_0}{p} \dfrac{c_1 - b_1 + 1}{2} \right) \right), \\
	\qquad\qquad\qquad\qquad\qquad\qquad\qquad\qquad\text{if $c_1$ is odd and $m_1$ is even,} \\
p^{2k + b_1 / 2} \left( 1 + p^{- (m_1 + 1)/2 + (c_1 - 1)/2} \left( \kronsym{c_0 m_0}{p} - \kronsym{-a b_0}{p} \right)  \right.\smallskip\\
	\quad\left.+ \dfrac{1}{p} \kronsym{-a b_0}{p} + \left( 1 - \dfrac{1}{p} \right) \left( \dfrac{c_1 - b_1 - 1}{2} + \kronsym{- a b_0}{p} \dfrac{c_1 - b_1 + 1}{2} \right) \right), \\
	\qquad\qquad\qquad\qquad\qquad\qquad\qquad\qquad\text{if $c_1$ and $m_1$ are odd,}
\end{cases}
\end{align*}
if $b_1$ is even and 
\begin{align*}
r_{p^k}(m, Q) &= 
\begin{cases}
p^{2k + (b_1 - 1)/ 2} \left( 1 + \kronsym{-a c_0}{p} - p^{- m_1/2 + c_1/2} \left( 1 + \dfrac{1}{p} \right) \kronsym{- a c_0}{p} \right) , \\
	\qquad\qquad\qquad\qquad\qquad\qquad\qquad\qquad\text{if $c_1$ and $m_1$ are even,} \\
 p^{2k + (b_1 - 1)/ 2} \left( 1 + \kronsym{-a c_0}{p} +  p^{- (m_1 + 1)/2 + c_1/2} \left( \kronsym{b_0 m_0}{p} - \kronsym{-a c_0}{p} \right) \right), \\
	\qquad\qquad\qquad\qquad\qquad\qquad\qquad\qquad\text{if $c_1$ is even and $m_1$ is odd,} \\
p^{2k + (b_1 - 1)/ 2} \left( 1 + \kronsym{- b_0 c_0}{p}  + p^{- m_1/2 + (c_1 - 1)/2} \left( \kronsym{a m_0}{p} - \kronsym{- b_0 c_0}{p} \right)  \right) , \\
	\qquad\qquad\qquad\qquad\qquad\qquad\qquad\qquad\text{if $c_1$ is odd and $m_1$ is even,} \\
p^{2k + (b_1 - 1)/ 2} \left( 1 + \kronsym{- b_0 c_0}{p} - p^{(-m_1 +  c_1)/2} \left( 1 + \dfrac{1}{p} \right) \kronsym{- b_0 c_0}{p} \right) , \\
	\qquad\qquad\qquad\qquad\qquad\qquad\qquad\qquad\text{if $c_1$ and $m_1$ are odd,}
\end{cases}
\end{align*}
if $b_1$ is odd.
By dividing by $p^{2k}$ and computing the limit as $k \to \infty$, we obtain \eqref{eq:density m1>=c1, b1 even} and \eqref{eq:density m1>=c1, b1 odd} if $m_1 \ge c_1$.

\subsection{Proof of Formulas for $\alpha_p (0, Q)$} 
We now compute $\alpha_p (0, Q)$. Let $k$ be a positive integer. Towards computing $\alpha_p (0, Q)$, we compute $r_{p^k}(0, Q)$ and then take the appropriate limit. By \eqref{eq:rpkQm 0+other},
\begin{align*}
r_{p^k}(0 , Q) &=  p^{2k} + \frac{1}{p^k} \sum_{t=1}^{p^k - 1} \gausssum{a t}{p^k} \gausssum{b_0 p^{b_1} t}{p^k} \gausssum{c_0 p^{c_1} t}{p^k} .
\end{align*}
Let $t = t_0 p^\tau$, where $0 \le \tau \le k-1$ and $t_0 \in {(\ZZ/p^{k-\tau}\ZZ)}^*$. Then
\begin{align*}
r_{p^k}(0 , Q) &=  p^{2k} + \frac{1}{p^k} \sum_{\tau=0}^{k-1} s_{k,\tau} \\
	&=  p^{2k} + \frac{1}{p^k} \sum_{\tau=0}^{k - c_1 -1} s_{k,\tau} + \frac{1}{p^k} \sum_{\tau = k - c_1}^{k - b_1 - 1} s_{k,\tau} + \frac{1}{p^k} \sum_{\tau = k - b_1}^{k-1} s_{k,\tau} , \numberthis\label{eq:sktausum0}
\end{align*}
where
\begin{align*}
s_{k,\tau} &= \sum_{t_0 \in {(\ZZ/p^{k-\tau}\ZZ)}^*} \gausssum{a t_0 p^\tau}{p^k} \gausssum{b_0 t_0 p^{b_1+\tau}}{p^k} \gausssum{c_0 t_0 p^{c_1+\tau}}{p^k} .
\end{align*}

Doing calculations similar to those in Lemmas~\ref{lem:sktau taumintok-1} and \ref{lem:sktausum tau-n1tok-1}, we notice that the result of Lemma~\ref{lem:sktausum tau-n1tok-1} is applicable to the sum $\sum_{\tau = k - b_1}^{k-1} s_{k,\tau}$ in \eqref{eq:sktausum0} if we take $b_1 = \min(m_1, b_1)$. Therefore,
\begin{align*}
r_{p^k}(0 , Q) &=  p^{2k} + \frac{1}{p^k} \sum_{\tau=0}^{k - c_1 -1} s_{k,\tau} + \frac{1}{p^k} \sum_{\tau = k - c_1}^{k - b_1 - 1} s_{k,\tau} + p^{2k}  ( p^{\lfloor b_1 / 2 \rfloor} -1) \\
	&=  \frac{1}{p^k} \sum_{\tau=0}^{k - c_1 -1} s_{k,\tau} + \frac{1}{p^k} \sum_{\tau = k - c_1}^{k - b_1 - 1} s_{k,\tau} + p^{2k + \lfloor b_1 / 2 \rfloor} .
\end{align*}
A quick check shows that the result of Lemma~\ref{lem:sktausum m1>=c1 tauk-c1tok-b1} is applicable to the sum $\sum_{\tau = k - c_1}^{k - b_1 - 1} s_{k,\tau}$ in \eqref{eq:sktausum0}, so
\begin{align*}
r_{p^k}(0, Q) &= 
\begin{cases}
\frac{1}{p^k} \sum_{\tau=0}^{k - c_1 -1} s_{k,\tau} + p^{2k + b_1 / 2} \smallskip\\
	\qquad+ p^{2k + b_1/2}\left( 1 - \dfrac{1}{p} \right) \left( \dfrac{c_1 - b_1}{2} + \dfrac{(-1)^{c_1} -1}{4} \right.\smallskip\\
	\qquad\qquad\qquad\qquad\qquad\left.+ \kronsym{- a b_0}{p} \left( \dfrac{c_1 - b_1}{2} + \dfrac{1 - (-1)^{c_1}}{4} \right) \right) , \\
	\qquad\qquad\qquad\qquad\qquad\qquad\qquad\qquad\text{if $b_1$ is even,} \\
\frac{1}{p^k} \sum_{\tau=0}^{k - c_1 -1} s_{k,\tau} + p^{2k + (b_1 - 1)/ 2} , \qquad\quad\text{if $b_1$ is odd.} 
\end{cases}
\end{align*}
Doing calculations similar to those in Lemmas~\ref{lem:sktau m1>=c1 tauk-m1tok-c1} and \ref{lem:sktausum m1>=c1 tauk-m1tok-c1}, we see that Lemma~\ref{lem:sktausum m1>=c1 tauk-m1tok-c1} is applicable to the sum $\sum_{\tau=0}^{k - c_1 -1} s_{k,\tau}$ in \eqref{eq:sktausum0} with $m_1 = k$. Thus, 
\begin{align*}
r_{p^k}(0, Q) &= 
\begin{cases}
p^{2k + b_1/2 - 1} \left( 1 - p^{-\lfloor k / 2\rfloor + c_1/2}\right) + p^{2k + b_1 / 2} \smallskip\\
	\qquad+ p^{2k + b_1/2}\left( 1 - \dfrac{1}{p} \right) \left( \dfrac{c_1 - b_1}{2} + \kronsym{- a b_0}{p} \dfrac{c_1 - b_1}{2} \right) , \\
	\qquad\qquad\qquad\qquad\qquad\qquad\qquad\qquad\text{if $b_1$ and $c_1$ are even,} \\
p^{2k + b_1/2 - 1} \kronsym{-a b_0}{p} \left( 1 - p^{- \lfloor (k + 1) / 2\rfloor + (c_1 + 1) /2} \right) + p^{2k + b_1 / 2} \smallskip\\
	\qquad+ p^{2k + b_1/2}\left( 1 - \dfrac{1}{p} \right) \left( \dfrac{c_1 - b_1 - 1}{2} + \kronsym{- a b_0}{p} \dfrac{c_1 - b_1 + 1}{2} \right) , \\
	\qquad\qquad\qquad\qquad\qquad\qquad\qquad\qquad\text{if $b_1$ is even and $c_1$ is odd,} \\
p^{2k + (b_1 - 1)/2} \kronsym{-a c_0}{p} \left( 1 - p^{- \lfloor (k + 1) / 2\rfloor + c_1/2 } \right) + p^{2k + (b_1 - 1)/ 2} , \\
	\qquad\qquad\qquad\qquad\qquad\qquad\qquad\qquad\text{if $b_1$ is odd and $c_1$ is even,} \\
p^{2k + (b_1 - 1)/2} \kronsym{- b_0 c_0}{p} \left( 1 - p^{- \lfloor k / 2\rfloor + (c_1 + 1)/2 - 1} \right) + p^{2k + (b_1 - 1)/ 2} , \\
	\qquad\qquad\qquad\qquad\qquad\qquad\qquad\qquad\text{if $b_1$ and $c_1$ are odd.} 
\end{cases}
\end{align*}
By dividing by $p^{2k}$ and taking the limit as $k \to \infty$, we obtain \eqref{eq:density0}.

\section*{Acknowledgments}
I thank Matthew Young and the National Science Foundation Research Experience for Undergraduates at Texas A\&M (funded under NSF DMS-1156589) for introducing me to local densities of quadratic forms. I thank Kenneth Williams for asking me to look at local densities of quadratic forms again. I also thank John Rickert, William Duke, Ranier Schulze-Pillot, Alex Kontorovich, Jonathan Hanke, and Gene Kopp for their helpful discussions. I thank Meinhard Peters for reading my paper and bringing Lomadze's work to my attention. I thank Kenneth Williams, Gene Kopp, Matthew Young, and Bruce Berndt for reading my paper and making great suggestions that improved the article. I especially thank the anonymous referee for thoroughly reading my manuscript and making excellent suggestions that improved the article. 

Much of this material is based upon work supported by the National Science Foundation Graduate Research Fellowship under NSF DGE-1842213.


\end{document}